\documentclass[12pt,reqno]{article}%{amsart}
\usepackage{geometry}
\usepackage[utf8]{inputenc}
\usepackage{amsfonts}
\usepackage{amsmath}
\usepackage{amsthm}
\usepackage{amssymb}
\usepackage{graphicx} 
\usepackage{tikz}
\usepackage{hyperref}

\newtheorem{definition}{Definition}[section]
\newtheorem{thm}[definition]{Theorem}
\newtheorem{prop}[definition]{Proposition}

\newtheorem{lem}[definition]{Lemma}

\newtheorem{ex}[definition]{Example}

\newcommand{\bz}{\mathbb Z}

\newcommand{\br}{\mathbb R}
\newcommand{\bc}{\mathbb C}

\newcommand{\calC}{\mathcal{C}}

\newcommand{\calS}{\mathcal{S}}

\providecommand{\keywords}[1]
{
  \small	
  \textbf{\textit{Keywords---}} #1
}

\title{A Short Proof for the Polynomiality\\ of the Stretched Littlewood-Richardson Coefficients}

\author{Warut Thawinrak\thanks{Department of Mathematics, University of California Davis, \texttt{wthawinrak@ucdavis.edu}}}
\date{}
\begin{document}

\maketitle
\thispagestyle{empty}

\begin{abstract}
    The stretched Littlewood-Richardson coefficient $c^{t\nu}_{t\lambda,t\mu}$ was conjectured by King, Tollu, and Toumazet to be a polynomial function in $t.$ It was shown to be true by Derksen and Weyman using semi-invariants of quivers. Later, Rassart used Steinberg's formula, the hive conditions, and the Kostant partition function to show a stronger result that $c^{\nu}_{\lambda,\mu}$ is indeed a polynomial in variables $\nu, \lambda, \mu$ provided they lie in certain polyhedral cones. Motivated by Rassart's approach, we give a short alternative proof of the polynomiality of $c^{t\nu}_{t\lambda,t\mu}$ using Steinberg's formula and a simple argument about the chamber complex of the Kostant partition function.
\end{abstract}


\section{Introduction}

The Littlewood-Richardson coefficients appear in many areas of mathematics \cite{fulton2000eigenvalues, littlewood1934group, macdonald1998symmetric, sagan2013symmetric, sundaramrep1990}. An example comes from the study of symmetric functions. The set of Schur functions $s_\lambda$, indexed by partitions $\lambda$, is a linear basis for the ring of symmetric functions. Thus, for any partitions $\lambda$ and $\mu$, the product of Schur functions $s_\lambda$ and $s_\mu$ can be uniquely expressed as
\begin{align}\label{eq:1}
    s_\lambda\cdot s_\mu = \sum_{\nu: |\nu| = |\lambda| + |\mu| } c^\nu_{\lambda,\mu} s_\nu
\end{align}
for some real numbers $c^\nu_{\lambda,\mu}$, where $|\lambda|$ denotes the sum of the parts of $\lambda$. The coefficient  $c^\nu_{\lambda,\mu}$ of $s_\nu$ in (\ref{eq:1}) is called the \textit{Littlewood-Richardson coefficient}.

There are several ways to compute $c^\nu_{\lambda,\mu}$ such as the Littlewood-Richardson rule \cite{Stanley1999}, the Littlewood-Richardson triangles \cite{PakVallejo2005}, the Berenstein-Zelevinsky triangles \cite{BZ1992}, and the honeycombs \cite{KnutsonTao1999}. In this paper, we employ the hive model that was first introduced by Knutson and Tao \cite{KnutsonTao1999}. The hive model imposes certain inequalities that allow us to compute $c^\nu_{\lambda,\mu}$ as the number of integer points in a rational polytope, which we call a hive polytope. 

For fixed partitions $\lambda, \mu, \nu$ such that $|\nu| = |\lambda| + |\mu|$, we define the \textit{stretched} Littlewood-Richardson coefficients to be the function $c^{t\nu}_{t\lambda,t\mu}$ for non-negative integers $t.$ The hive model implies that
\[c^{t\nu}_{t\lambda,t\mu} = \text{ the number of integer points in the } t^\mathrm{th}\text{-dilation of the hive polytope}.\]
By Ehrhart theory (see Thoerem \ref{Ehrhar}), $c^{t\nu}_{t\lambda,t\mu}$ is a quasi-polynolmial in $t \in \bz$, which means $c^{t\nu}_{t\lambda,t\mu}$ is a function of the form $a_d(t)t^d + \cdots + a_1(t)t + a_0(t)$ where each of $a_d(t), \dots, a_0(t)$ is a periodic function in $t$ with an integral period. The function $c^{t\nu}_{t\lambda,t\mu}$ was, however, observed and conjectured by King, Tollu, and Toumazet \cite{KingTolluTaumazet2003} to be a polynomial function in $t$ (as opposed to a quasi-polynomial). The conjecture was then shown to be true by Derksen-Weyman \cite{DerksenWeywan2002}, and Rassart \cite{Rassart2004}. More precisely, they proved the following theorem.   

\begin{thm}\label{main}
Let $\mu, \lambda, \nu$ be partitions with at most $k$ part such that $|\nu| = |\lambda| + |\mu|.$ Then
$c^{t\nu}_{t\lambda,t\mu}$ is a polynomial in $t$ of degree at most $\binom{k-1}{2}.$
\end{thm}

The proof by Derksen and Weyman \cite{DerksenWeywan2002} makes use of semi-invariants of quivers. They proved a result on the structure of a ring of quivers and then derived the polynomiality of $c^{t\nu}_{t\lambda,t\mu}$ as a special case.  Later, Rassart \cite{Rassart2004} proved a stronger result, which gives Theorem \ref{main} as an easy consequence, by showing that $c^{\nu}_{\lambda,\mu}$ is a polynomial in variables $\lambda, \mu, \nu$ provided that they lie in certain polyhedral cones of a chamber complex. The proof by Rassart employs Steinberg's formula, the hive conditions, and the Kostant partition function to give the chamber complex of cones in which $c^{\nu}_{\lambda,\mu}$ is a polynomial in variables $\lambda, \mu, \nu$. A considerably large portion of Rassart's paper was devoted to describing this chamber complex and showing its desired property, resulting in a fairly long justification. We note that although this chamber complex of cones was provided, it is in practice computationally hard to work out the cones.

Inspired by Rassart's approach, we ask if similar tools can be utilized to give a simple proof of Theorem \ref{main} directly. We found that Steinberg's formula and a simple argument about the chamber complex of the Kostant partition function are indeed sufficient. The main objective of this paper is to give a short alternative proof of Theorem \ref{main} using this idea. 

\section{Preliminaries}

We begin this section by presenting necessary notations and theories related to polytopes and then describe the hive model for computing $c^{\nu}_{\lambda,\mu}$. The hive model will help us understand the behavior of the stretched Littlewood-Richardson coefficients through a property of polytopes. We then introduce the Kostant partition functions and state Steinberg's formula and related results that will later be used for proving Theorem \ref{main}.

\subsection{Ehrhart Theory}

A \textit{polyhedron} $P$ in $\br^d$ is the solution to a finite set of linear inequalities, that is, 
\[P = \left\{(x_1, \dots, x_d) \in \br^d\,\Big\vert\, \sum^d_{j = 1}a_{ij}x_j \leq b_i \text{ for } i \in I\right\}\]
where $a_{ij} \in \br$, $b_i \in \br,$ and $I$ is a finite set of indices. A \textit{polytope} is a bounded polyhedron. We can also equivalently define a polytope in $\br^d$ as the convex hull of finitely many points in $\br^d.$ A polytope is said to be \textit{rational} if all of its vertices have rational coordinates, and is said to be \textit{integral} if all of its vertices have integral coordinates. We refer readers to \cite{Ziegler1995} for basic definitions regarding polyhedra.

For a polytope $P$ in $\br^d$ and a non-negative integer $t,$ the $t^{\mathrm{th}}$-dilation $tP$ is the set $\{tx \,|\, x \in P\}.$ We define
\[i(P,t) := |\bz^d \cap tP|\]
to be the number of integer points in the $t^\mathrm{th}$-dilation $tP.$ 

Recall that a quasi-polynomial is a function of the form $f(t) = a_d(t)t^d + \cdots a_1(t)t + a_0(t)$ where each of $a_d(t), \dots, a_0(t)$ is a periodic function in $t$ with an integral period. The \textit{period} of $f(t)$ is the least common period of $a_d(t), \dots, a_0(t)$. Clearly, a quasi-polynomial of period one is a polynomial.

For a rational polytope $P$, the least common multiple of the denominators of the coordinates of its vertices is called the \textit{denominator} of $P$.  The behavior of the function $i(P,t)$ is described by the following theorem due to Ehrhart \cite{Ehrhart1962}.

\begin{thm}[Ehrhart Theory]\label{Ehrhar}
If $P$ is a rational polytope, then $i(P,t)$ is a quasi-polynomial in $t.$ Moreover, the period of $i(P,t)$ is a divisor of the denominator of $P.$ In particular, if $P$ is an integral polytope, then $i(P,t)$ is a polynomial in $t.$
\end{thm}

The polynomial (resp. quasi-polynomial) $i(P,t)$ is called the \textit{Ehrhart polynomial} of $P$ (resp. Ehrhart quasi-polynomial of $P$).

\subsection{The Littlewood-Richardson Coefficients}

We say that $\lambda = (\lambda_1, \dots, \lambda_k)$ is a \textit{partition} of a non-negative integer $m$ if $\lambda_1 \geq \cdots \geq \lambda_k$ are positive integers such that $\lambda_1 + \cdots + \lambda_k = m.$ For convenience, we will abuse the notation by allowing $\lambda_i$ to be zero. The positive numbers among $\lambda_1, \dots, \lambda_k$ are called \textit{parts} of $\lambda$. For example, $\lambda = (2,2,1,0)$ is a partition of $5$ with $3$ parts. We write $|\lambda|$ to denote $\lambda_1+ \cdots + \lambda_k.$

A \textit{hive} $\Delta_k$ of size $k$ is an array of vertices $h_{ij}$ arranged in a triangular grid consisting of $k^2$ small equilateral triangles as shown in Figure \ref{fig1}. Two adjacent equilateral triangles form a rhombus with two equal obtuse angles and two equal acute angles. There are three types of these rhombi: tilted to the right, left, and vertical as shown in Figure \ref{fig1}.

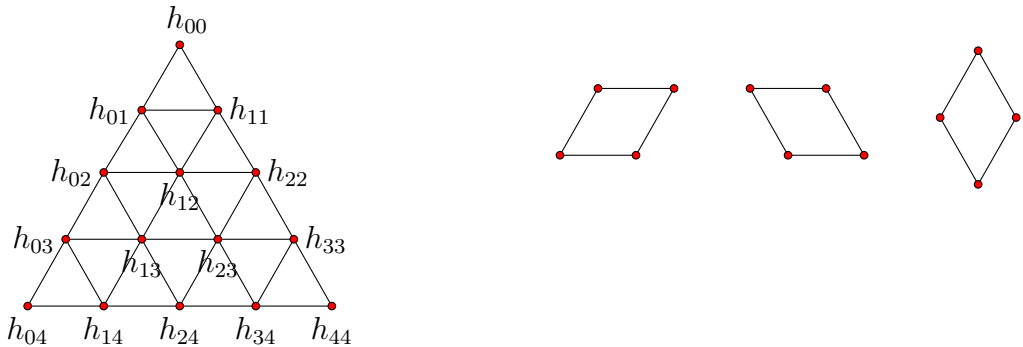
\begin{figure}[h]
  \centering
\begin{tikzpicture}
\draw (0,0) node[below] {$h_{04}$} -- (1,0) node[below] {$h_{14}$} -- (2,0) node[below] {$h_{24}$} -- (3,0) node[below] {$h_{34}$} -- (4,0) node[below] {$h_{44}$};
\draw (0.5,0.886) node[left] {$h_{03}$} -- (1.5, 0.886) node[below] {$h_{13}$} -- (2.5, 0.886) node[below] {$h_{23}$} -- (3.5,0.886) node[right] {$h_{33}$};
\draw (1,1.772) node[left] {$h_{02}$} -- (2, 1.772) node[below] {$h_{12}$} -- (3,1.772) node[right] {$h_{22}$};
\draw (1.5,2.598) node[left] {$h_{01}$} -- (2.5,2.598) node[right] {$h_{11}$};
\draw (0,0) -- (0.5,0.886) -- (1,1.772) -- (1.5,2.598) --(2,3.464);
\draw (1,0) -- (1.5,0.886) -- (2,1.772) -- (2.5,2.598);
\draw (2,0) -- (2.5,0.886) -- (3,1.772);
\draw (3,0) -- (3.5,0.886);
\draw (4,0) -- (3.5,0.886) -- (3,1.772) -- (2.5,2.598) --(2,3.464);
\draw (3,0) -- (2.5,0.886) -- (2,1.772) -- (1.5,2.598);
\draw (2,0) -- (1.5,0.886) -- (1,1.772);
\draw (1,0) -- (0.5,0.886);
\node at (2.1,3.8) {$h_{00}$};
\draw (7.5, 2.886) -- (7,2) -- (8,2) -- (8.5, 2.886) -- (7.5, 2.886); 
\draw[fill = red] (7.5, 2.886) circle (0.5mm);
\draw[fill = red] (7,2) circle (0.5mm);
\draw[fill = red] (8,2) circle (0.5mm);
\draw[fill = red] (8.5, 2.886) circle (0.5mm);

\draw (9.5, 2.886) -- (10.5, 2.886) -- (11,2) -- (10,2) -- (9.5, 2.886); 
\draw[fill = red] (9.5, 2.886) circle (0.5mm);
\draw[fill = red] (10.5, 2.886) circle (0.5mm);
\draw[fill = red] (10,2) circle (0.5mm);
\draw[fill = red] (11,2) circle (0.5mm);

\draw (12,2.5) -- (12.5,3.386) -- (13,2.5) -- (12.5,1.614) -- (12,2.5); 
\draw[fill = red] (12,2.5) circle (0.5mm);
\draw[fill = red] (12.5,3.386) circle (0.5mm);
\draw[fill = red] (13,2.5) circle (0.5mm);
\draw[fill = red] (12.5,1.614) circle (0.5mm);

\draw[fill = red] (0,0) circle (0.5mm);
\draw[fill = red] (1,0) circle (0.5mm);
\draw[fill = red] (2,0) circle (0.5mm);
\draw[fill = red] (3,0) circle (0.5mm);
\draw[fill = red] (4,0) circle (0.5mm);
\draw[fill = red] (0.5,0.886) circle (0.5mm);
\draw[fill = red] (1.5,0.886) circle (0.5mm);
\draw[fill = red] (2.5,0.886) circle (0.5mm);
\draw[fill = red] (3.5,0.886) circle (0.5mm);
\draw[fill = red] (1,1.772) circle (0.5mm);
\draw[fill = red] (2,1.772) circle (0.5mm);
\draw[fill = red] (3,1.772) circle (0.5mm);
\draw[fill = red] (1.5,2.598) circle (0.5mm);
\draw[fill = red] (2.5,2.598) circle (0.5mm);
\draw[fill = red] (2,3.464) circle (0.5mm);
\end{tikzpicture}
 \caption{Hive of size 4 (left), and the three types of rhombi in a hive (right)}
 \label{fig1}
\end{figure}

Let $\lambda = (\lambda_1, \dots, \lambda_k), \mu = (\mu_1, \dots, \mu_k), \nu = (\nu_1, \dots, \nu_k)$ be partitions with at most $k$ parts such that $|\nu| = |\lambda| + |\mu|.$ A \textit{hive of type} $(\nu, \lambda, \mu)$ is a labelling $(h_{ij})$ of $\Delta_k$ that satisfies the following \textit{hive conditions}.
\begin{enumerate}
    \item[(HC1)]\label{HC1} [Boundary condition] The labelings on the boundary are determined by $\lambda, \mu, \nu$ in the following ways.
    \begin{align*}
    h_{00} = 0, \ h_{j j} - h_{j-1 j-1} &= \nu_j, \ h_{0j} - h_{0j-1} = \lambda_j, &\text{ for } 1 \leq j \leq k.\\
        h_{ik} - h_{i-1k} &= \mu_i, &\text{ for } 1 \leq i \leq k.
    \end{align*}
    \item[(HC2)]\label{HC2} [Rhombi condition]  For every rhombus, the sum of the labels at obtuse vertices is greater than or equal to the sum of the labels at acute vertices. That is, for $1 \leq i < j \leq k,$
    \begin{align*}
    h_{ij} - h_{ij-1} &\geq h_{i-1 j} - h_{i-1 j-1},\\
    h_{ij} - h_{i-1 j} &\geq h_{i+1 j+1} - h_{ij+1}, \text{ and }\\
    h_{i-1 j} - h_{i-1 j-1} &\geq h_{ij+1} - h_{ij}.
    \end{align*}
\end{enumerate}

Let $H_k(\nu,\lambda,\mu)$ denote the set of all hive of type $(\nu, \lambda, \mu)$. Then the hive conditions (HC1) and (HC2) imply that $H_k(\nu,\lambda,\mu)$ is a rational polytope in $\br^{n}$ where $n = \binom{k+2}{2}$. Hence, we will call $H_k(\nu,\lambda,\mu)$ the \textit{hive polytope of type} $(\nu,\lambda,\mu)$. Knutson-Tao \cite{KnutsonTao1999} and Buch \cite{Buch2000} showed that 
\[c^{\nu}_{\lambda,\mu} = \text{ the number of integer points in } H_k(\nu,\lambda,\mu).\]

\begin{ex}
Let $k = 3$, $\nu = (4,3,1), \lambda = (2,1,0),$ and $\mu = (3,2,0)$, we have that $c^{\nu}_{\lambda,\mu} = 2$. The two corresponding integer points (integer labels) of $H_{3}(\nu, \lambda, \mu)$ are shown in Figure \ref{fig2}.

\begin{figure}[h]
  \centering
\begin{tikzpicture}
\node[left] at (0.5,0.886) {$3$};
\node[below] at (1.5, 0.886){$6$}; 
\node[below] at (2.5, 0.886) {$8$}; 
\node[right] at (3.5,0.886) {$8$};
\node[left] at (1,1.772) {$3$};
\node[below] at (2, 1.772) {$5$}; 
\node[right] at (3,1.772) {$7$};
\node[left] at (1.5,2.598) {$2$};
\node[right] at (2.5,2.598) {$4$};
\node at (2.1,3.8) {$0$};

\draw[fill = red] (0.5,0.886) circle (0.5mm);
\draw[fill = red] (1.5,0.886) circle (0.5mm);
\draw[fill = red] (2.5,0.886) circle (0.5mm);
\draw[fill = red] (3.5,0.886) circle (0.5mm);
\draw[fill = red] (1,1.772) circle (0.5mm);
\draw[fill = red] (2,1.772) circle (0.5mm);
\draw[fill = red] (3,1.772) circle (0.5mm);
\draw[fill = red] (1.5,2.598) circle (0.5mm);
\draw[fill = red] (2.5,2.598) circle (0.5mm);
\draw[fill = red] (2,3.464) circle (0.5mm);

\node[left] at (5.5,0.886) {$3$};
\node[below] at (6.5, 0.886){$6$}; 
\node[below] at (7.5, 0.886) {$8$}; 
\node[right] at (8.5,0.886) {$8$};
\node[left] at (6,1.772) {$3$};
\node[below] at (7, 1.772) {$6$}; 
\node[right] at (8,1.772) {$7$};
\node[left] at (6.5,2.598) {$2$};
\node[right] at (7.5,2.598) {$4$};
\node at (7.1,3.8) {$0$};

\node at (11,2.5) {$\nu = (4,3,1)$};
\node at (11,2) {$\lambda = (2,1,0)$};
\node at (11,1.5) {$\mu = (3,2,0)$};

\draw[fill = red] (5.5,0.886) circle (0.5mm);
\draw[fill = red] (6.5,0.886) circle (0.5mm);
\draw[fill = red] (7.5,0.886) circle (0.5mm);
\draw[fill = red] (8.5,0.886) circle (0.5mm);
\draw[fill = red] (6,1.772) circle (0.5mm);
\draw[fill = red] (7,1.772) circle (0.5mm);
\draw[fill = red] (8,1.772) circle (0.5mm);
\draw[fill = red] (6.5,2.598) circle (0.5mm);
\draw[fill = red] (7.5,2.598) circle (0.5mm);
\draw[fill = red] (7,3.464) circle (0.5mm);
\end{tikzpicture}
 \caption{The only two integer points (integer labels) of $H_3(\nu, \lambda, \mu)$}
 \label{fig2}
\end{figure}
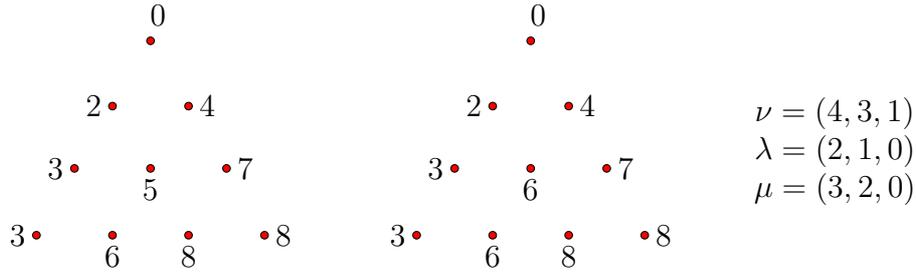
\end{ex} 
For fixed partitions $\lambda, \mu, \nu$ with at most $k$ parts such that $|\nu| = |\lambda| + |\mu|$, we define the the \textit{stretched} Littlewood-Richardson coefficient to be the function $c^{t\nu}_{t\lambda,t\mu}$ for non-negative integer $t.$ Because $H_k(t\nu,t\lambda,t\mu) = tH_k(\nu,\lambda,\mu)$, we have that
\[c^{t\nu}_{t\lambda,t\mu} = i(H_k(\nu,\lambda,\mu),t).\]

Examples provided in \cite{KingTolluTaumazet2003} indicate that $H_k(\nu,\lambda,\mu)$ is in general not an integral polytope. Thus, by Ehrhart theory (Theorem \ref{Ehrhar}), $c^{t\nu}_{t\lambda,t\mu}$ is a quasi-polynomial in $t$. We will show that $c^{t\nu}_{t\lambda,t\mu}$ is indeed a polynomial in $t$ even though the corresponding hive polytope $H_k(\nu,\lambda,\mu)$ is not integral.

\subsection{Kostant Partition Function and Steinberg's Formula} 

We will show the polynomiality of $c^{t\nu}_{t\lambda,t\mu}$ by using Steinberg's formula as derived in \cite{Rassart2004} by Rassart and the chamber complex of the Kostant partition function. To this end, we state the related notations and results for later reference.

Let $e_1, \dots, e_k$ be the standard basis vectors in $\br^k$, and let $\Delta_+ = \{e_i - e_j: 1 \leq i < j \leq k\}$ be the set of positive roots of the root system of type $A_{k-1}.$ We define $M$ to be the matrix whose columns consist of the elements of $\Delta_+$. The \textit{Kostant partition function} for the root system of type $A_{k-1}$ is the function $K: \bz^k \longrightarrow \bz_{\geq 0}$ defined by
\[K(v) = \Big\vert\left\{b \in \bz_{\geq 0}^{\binom{k}{2}}\,|\, Mb = v\right\}\Big\vert.\]
That is, $K(v)$ equals the number of ways to write $v$ as nonnegative integer linear combinations of the positive roots in $\Delta_+$. 

An important property of the matrix $M$, when written in the basis of simple roots $\{e_i - e_{i+1}\,|\, i = 1, \dots, k-1\}$, is that it is totally unimodular, i.e. the determinant of every square submatrix equals $-1, 0,$ or $1.$ Indeed, it is shown in \cite{Schrijver1986} that a matrix $A$ is totally unimodular if every column of $A$ only consists of 0’s and 1’s in a way that the 1’s come in a consecutive block. Let 
\[\mathrm{cone}(\Delta_+) = \left\{\sum \lambda_v v\,|\, v \in \Delta_+, \lambda_v \geq 0\right\}\]
be the cone spanned by the vectors in $\Delta_+.$ The \textit{chamber complex} is the polyhedral subdivision of $\mathrm{cone}(\Delta_+)$ that is obtained from the common refinement of cones $\mathrm{cone}(B)$ where $B$ are the maximum linearly independent subsets of $\Delta_+$. A maximum cell (a cone of maximum dimension) $\calC$ in the chamber complex is called a \textit{chamber}. Since $M$ is totally unimodular, the behavior of $K(v)$ is given by the following lemma as a special case of \cite[Theorem 1]{Sturmfels1995} due to Sturmfels.

\begin{lem}\label{Sturmfels}
    Let $\calC$ be a chamber in the chamber complex of $\mathrm{cone}(\Delta_+)$. Then the Kostant partition function $K(v)$ is a polynomial in $v = (v_1, \dots, v_k)$ on $\calC$ of degree at most $\binom{k-1}{2}$.
\end{lem}

Steinberg's formula \cite{Steinberg1961} expresses the tensor product of two irreducible representations of semisimple Lie algebras as the direct sum of other irreducible representations. When restricting the formula to $\mathrm{SL}_k\bc$, we obtain the following version of Steinberg's formula for computing $c^{\nu}_{\lambda,\mu}$.

\begin{thm}[Steinberg's Formula]\label{Steinberg} Let $\mu, \lambda, \nu$ be partitions with at most $k$ part such that $|\nu| = |\lambda| + |\mu|.$ Then
\[c^{\nu}_{\lambda,\mu} = \sum_{\sigma,\tau \in \calS_k}(-1)^{\mathrm{inv}(\sigma\tau)}K(\sigma(\lambda +\delta) + \tau(\mu + \delta) - (\nu + 2\delta))\]
where $\mathrm{inv}(\psi)$ is the number of inversions of the permutation $\psi$ and 
\[\delta = \frac{1}{2}\sum_{1 \leq i < j \leq k}(e_i - e_j) = \frac{1}{2}(k-1, k-3, \dots, -(k-3), -(k-1))\]
is the Weyl vector for type $A_{k-1}.$
\end{thm}

Details of the derivation can be found in \cite[section 1.1]{Rassart2004}. 

\section{Proof of the Polynomiality}

We are now ready to prove Theorem \ref{main}.

\begin{proof}[Proof of Theorem \ref{main}] The hive conditions imply that $c^{t\nu}_{t\lambda, t\mu}$ is a quasi-polynomial in $t.$ To see that $c^{t\nu}_{t\lambda, t\mu}$ is in fact a polynomial in $t$, it suffices to show that there exists an integer $N$ such that $c^{t\nu}_{t\lambda, t\mu}$ is a polynomial in $t$ for $t \geq N.$

For $\sigma, \tau \in \calS_k$, let 
\begin{align*}
    r^{\lambda,\mu,\nu}_{\sigma,\tau}(t) &:= \sigma(t\lambda + \delta) + \tau(t\mu+\delta) - (t\nu +2\delta)\\
    &= t(\sigma(\lambda) +\tau(\mu) - \nu) + \sigma(\delta)+\tau(\delta)-2\delta.
\end{align*}
Then $r^{\lambda,\mu,\nu}_{\sigma,\tau}(t)$ is a ray (when allowing $t$ to be non-negative real number) emanating from $\sigma(\delta)+\tau(\delta)-2\delta$ in the direction of $\sigma(\lambda) +\tau(\mu) - \nu$.

By Steinberg's formula, 
\[c^{t\nu}_{t\lambda, t\mu} = \sum_{\sigma, \tau \in \calS_k}(-1)^{\mathrm{inv}(\sigma\tau)}K(r^{\lambda,\mu,\nu}_{\sigma,\tau}(t)).\]

Lemma \ref{Sturmfels} states that $K(v)$ is a polynomial in $v$ when $v$ stays in one particular cone (chamber) of the chamber complex of $\mathrm{cone}(\Delta_+)$. Because there are only finitely many cones in the chamber complex, we have that for every pair $\sigma, \tau \in \calS_k$ there exists an integer $N^{\lambda,\mu,\nu}_{\sigma,\tau}$ such that exactly one of the following happens:
\begin{enumerate}
\item[(1)] The ray $r^{\lambda,\mu,\nu}_{\sigma,\tau}(t)$ lies in one particular cone of the chamber complex for all $t \geq N^{\lambda,\mu,\nu}_{\sigma,\tau}$
\item[(2)] The ray $r^{\lambda,\mu,\nu}_{\sigma,\tau}(t)$ lies outside $\mathrm{cone}(\Delta_+)$ for all $t \geq N^{\lambda,\mu,\nu}_{\sigma,\tau}.$
\end{enumerate}
If (1) is satisfied, then $K(r^{\lambda,\mu,\nu}_{\sigma,\tau}(t))$ is a polynomial in $t$ for $t \geq  N^{\lambda,\mu,\nu}_{\sigma,\tau}$. If (2) is satisfied, then $K(r^{\lambda,\mu,\nu}_{\sigma,\tau}(t))$ is the zero polynomial for $t \geq  N^{\lambda,\mu,\nu}_{\sigma,\tau}.$ In either case, $K(r^{\lambda,\mu,\nu}_{\sigma,\tau}(t))$ is a polynomial in $t$ for $t \geq  N^{\lambda,\mu,\nu}_{\sigma,\tau}$.  Now let
\[N = \max_{\sigma, \tau \in \calS_k}\{ N^{\lambda,\mu,\nu}_{\sigma,\tau}\}.\]
Then Steinberg's formula implies that $c^{t\nu}_{t\lambda, t\mu}$ is a polynomial in $t$ for $t \geq N.$ Therefore, $c^{t\nu}_{t\lambda, t\mu}$ is a polynomial in $t$. 

By Lemma \ref{Sturmfels}, each polynomial piece of $K(v)$ has degree at most $\binom{k-1}{2}$. Thus, for every $\sigma, \tau$, we have that $K(r^{\lambda,\mu,\nu}_{\sigma,\tau}(t))$ is a polynomial in $t$ of degree at most $\binom{k-1}{2}$ for $t \geq  N^{\lambda,\mu,\nu}_{\sigma,\tau}.$ Hence, $c^{t\nu}_{t\lambda, t\mu}$ is a polynomial in $t$ of degree at most $\binom{k-1}{2}$. 
\end{proof}

In the proof of Theorem \ref{main}, we showed that every $K(r^{\lambda,\mu,\nu}_{\sigma,\tau}(t))$ is eventually either the zero polynomial or a non-zero polynomial in $t$. Proposition \ref{nonzero} gives a characterization of those $K(r^{\lambda,\mu,\nu}_{\sigma,\tau}(t))$ that eventually become non-zero polynomials. The proof uses the following characterization of non-zero $K(v)$.

\begin{lem}\label{characterization}
Let $v = (v_1, \dots, v_k)$ be a vector in $\bz^k$ with $v_1 + \cdots + v_k = 0$. Then $K(v)$ is non-zero if and only if $v_1 + \cdots + v_i \geq 0$ for all $i = 1, \dots, k$.
\end{lem}

\begin{proof}
    Let $M^*$ be the matrix $M$ written using the simple roots $e_1 - e_{2}, \dots, e_{k-1} - e_k$ as a basis. Then, the entries of $M^*$ are only $0$ and $1$. Moreover, because the simple roots themselves are columns of $M$, we have that the identity matrix is a submatrix of $M^*$. Similarly, let $v^*$ be the vector $v$ written using the simple roots as a basis. Then, $v^* = (v_1, v_1+v_2, \dots, v_1 + \cdots v_{k-1})$. The desired result is obtained by observing that 
    \[K(v) = \Big\vert\left\{b \in \bz_{\geq 0}^{\binom{k}{2}}\,|\, M^*b = v^*\right\}\Big\vert.\]
\end{proof}

\begin{prop}\label{nonzero}
Let $\mu, \lambda, \nu$ be partitions with at most $k$ part such that $|\nu| = |\lambda| + |\mu|.$ For $\sigma, \tau \in \calS_k,$ let 
\[r^{\lambda,\mu,\nu}_{\sigma,\tau}(t) = t\beta + \gamma\]
where $\beta = \sigma(\lambda) +\tau(\mu) - \nu$ and $\gamma = \sigma(\delta)+\tau(\mu)-2\delta.$
Then there exists an integer $N^{\lambda,\mu,\nu}_{\sigma,\tau}$ such that  $K(r^{\lambda,\mu,\nu}_{\sigma,\tau}(t))$ is a non-zero polynomial in $t$ for $t \geq N^{\lambda,\mu,\nu}_{\sigma,\tau}$ if and only if for all $i = 1, \dots, k$ we have that
\begin{enumerate}
    \item[(1)] $\beta_1 + \beta_2 + \cdots + \beta_i$ is positive, or
    \item[(2)] $\beta_1 + \beta_2 + \cdots + \beta_i$ is zero and $\gamma_1 + \gamma_2 + \cdots + \gamma_i$ is non-negative. 
\end{enumerate} 
\end{prop} 

\begin{proof}

Let $r^{\lambda,\mu,\nu}_{\sigma,\tau}(t) = (r_1(t), \dots, r_k(t))$. Then $r_i(t) = t\beta_i + \gamma_i$. In the proof of Theorem \ref{main}, we showed that there exists a positive integer $N^{\lambda,\mu,\nu}_{\sigma,\tau}$ such that  $K(r^{\lambda,\mu,\nu}_{\sigma,\tau}(t))$ is a polynomial in $t$ for $t \geq N^{\lambda,\mu,\nu}_{\sigma,\tau}.$ For every $i = 1, \dots, k$, the partial sum $r_1(t) + \cdots + r_i(t)$ is non-negative for all $t \geq N^{\lambda,\mu,\nu}_{\sigma,\tau}$ precisely when one of the two conditions meets for all $i = 1, \dots, k$. Thus, by Lemma \ref{characterization}, $K(r^{\lambda,\mu,\nu}_{\sigma,\tau}(t))$ is a non-zero polynomial for $t \geq N^{\lambda,\mu,\nu}_{\sigma,\tau}.$
\end{proof}

\newpage

\section*{Acknowledgement}
I am grateful to Fu Liu. Her careful review and thoughtful comments significantly improved the exposition of this paper. I also would like to thank UC Davis's College of Letter and Science for providing the Dean’s Summer Graduate Fellowship to support me during the summer of 2022.

\section*{Data Availability Statement}
Data sharing not applicable to this article as no datasets were generated or analysed during the current study.

\section*{Statements and Declarations}
\noindent\textbf{Competing Interests:} The author declares that there is no conflict of interest.

%% if you use biblatex then this generates the bibliography
%% if you use some other method then remove this and do it your own way
\bibliographystyle{abbrv}
\bibliography{BibContainer}

\begin{thebibliography}{10}

\bibitem{BZ1992}
A.~D. Berenstein and A.~V. Zelevinsky.
\newblock Triple multiplicities for $\mathrm{sl}(r+ 1)$ and the spectrum of the
  exterior algebra of the adjoint representation.
\newblock {\em Journal of Algebraic Combinatorics}, 1(1):7--22, 1992.

\bibitem{Buch2000}
A.~S. Buch.
\newblock The saturation conjecture (after a. knutson and t. tao), with an
  appendix by william fulton.
\newblock {\em Enseign. Math.}, 46(1/2):43--60, 2000.

\bibitem{DerksenWeywan2002}
H.~Derksen and J.~Weyman.
\newblock On the littlewood--richardson polynomials.
\newblock {\em Journal of Algebra}, 255(2):247--257, 2002.

\bibitem{Ehrhart1962}
E.~Ehrhart.
\newblock Sur les polyèdres rationnels homothétiques à n dimensions.
\newblock {\em C. R. Acad. Sci. Paris}, 254:616--618, 1962.

\bibitem{fulton2000eigenvalues}
W.~Fulton.
\newblock Eigenvalues, invariant factors, highest weights, and schubert
  calculus.
\newblock {\em Bulletin of the American Mathematical Society}, 37(3):209--249,
  2000.

\bibitem{KingTolluTaumazet2003}
R.~C. King, C.~Tollu, and F.~Toumazet.
\newblock Stretched littlewood-richardson and kostka coefficients.
\newblock In {\em CRM Proceedings and Lecture Notes Vol. 34}, pages 99--112.
  American Mathematical Society, 2004.

\bibitem{KnutsonTao1999}
A.~Knutson and T.~Tao.
\newblock The honeycomb model of $gl_n(\bc)$ tensor products i: Proof of the
  saturation conjecture.
\newblock {\em Journal of the American Mathematical Society}, 12(4):1055--1090,
  1999.

\bibitem{littlewood1934group}
D.~E. Littlewood and A.~R. Richardson.
\newblock Group characters and algebra.
\newblock {\em Philosophical Transactions of the Royal Society of London.
  Series A, Containing Papers of a Mathematical or Physical Character},
  233(721-730):99--141, 1934.

\bibitem{macdonald1998symmetric}
I.~G. Macdonald.
\newblock {\em Symmetric functions and Hall polynomials}.
\newblock Oxford university press, 1998.

\bibitem{PakVallejo2005}
I.~Pak and E.~Vallejo.
\newblock Combinatorics and geometry of littlewood--richardson cones.
\newblock {\em European Journal of Combinatorics}, 26(6):995--1008, 2005.

\bibitem{Rassart2004}
E.~Rassart.
\newblock A polynomiality property for littlewood--richardson coefficients.
\newblock {\em Journal of Combinatorial Theory, Series A}, 107(2):161--179,
  2004.

\bibitem{sagan2013symmetric}
B.~E. Sagan.
\newblock {\em The symmetric group: representations, combinatorial algorithms,
  and symmetric functions}, volume 203.
\newblock Springer Science \& Business Media, second edition edition, 2001.

\bibitem{Schrijver1986}
A.~Schrijver.
\newblock {\em Theory of linear and integer programming}.
\newblock John Wiley \& Sons, 1998.

\bibitem{Stanley1999}
R.~P. Stanley.
\newblock {\em Enumerative Combinatorics, volume 2}.
\newblock Cambridge Studies in Advanced Mathematics. Cambridge University
  Press, 1999.

\bibitem{Steinberg1961}
R.~Steinberg.
\newblock A general clebsch-gordan theorem.
\newblock {\em Bulletin of the American Mathematical Society}, 67(4):406--407,
  1961.

\bibitem{Sturmfels1995}
B.~Sturmfels.
\newblock On vector partition functions.
\newblock {\em Journal of combinatorial theory, series A}, 72(2):302--309,
  1995.

\bibitem{sundaramrep1990}
S.~Sundaram.
\newblock Tableaux in the representation theory of the classical lie groups.
\newblock {\em Institute for Mathematics and Its Applications}, 19:191--225, 01
  1990.

\bibitem{Ziegler1995}
G.~Ziegler.
\newblock {\em Lectures on Polytopes}, volume 152 of {\em Graduate Texts in
  Mathematics}.
\newblock Springer New York, NY, 1996.

\end{thebibliography}

\end{document}